\documentclass{amsart}
\usepackage{amssymb}  
\usepackage{amsmath} 
\usepackage{amsthm}  
\usepackage{bm} 
\usepackage{stmaryrd} 

\advance\hoffset by -1in
\advance\voffset by -1in

\newtheorem{thm}{Theorem}
\newtheorem{lem}[thm]{Lemma}

\newtheorem{cor}[thm]{Corollary}

\title{On a finite $2,3$-generated group of period $12$}

\author{Andrei V. Zavarnitsine}
\address{Sobolev Institute of Mathematics\\
4, Koptyug av.\\
630090, Novosibirsk\\
Russia}
\email{zav@math.nsc.ru}

\thanks{Supported by the Russian Foundation for Basic Research
(project 13--01--00505).}

\begin{document}

\begin{abstract}
Using calculations in computer algebra systems along with some theoretic results, we construct the largest finite group of period $12$ generated by an element of order $2$ and an element of order~$3$. In particular, we prove that this group has order $2^{66}\cdot3^7$.

{\sc Keywords:} periodic groups, Burnside problem \\
{\em 2010 MSC:}
   \ 20F05  
\ 20F50  
\ 20-04  

\end{abstract}

\maketitle

\section{Introduction}

A group $G$ has period $n$ if $x^n=1$ for all $x\in G$ or, equivalently, if $\mathrm{exp}(G)$ is finite and  divides $n$. Groups of period $12$ are of interest in light of the Burnside problem.

There has been a recent progress in proving local finiteness of certain classes of groups of period $12$. For example, groups of period $12$ in which the product of every two involutions has order distinct from $6$ (respectively, from $4$) are locally finite by \cite{lmm} and \cite{mm}. Groups of period $12$ without elements of order $12$ are locally finite by \cite{mam}.

A group is $2,3$-{\em generated} if it is a quotient of the free product $\mathbb{Z}_2*\mathbb{Z}_3$. A $2,3$-generated group of period $12$ will be called a $(2,3;\,12)$-{\em group}. It is not known if the free $(2,3;\,12)$-group $B$ is finite, and our aim is to study the finite quotients of $B$. The importance of this study lies in the fact that it represents the smallest unknown case among the $2$-generator groups of period $12$.

From the positive solution of the restricted Burnside problem, it easily follows that there exists a unique maximal finite $(2,3;\,12)$-group, which we denote by $B_0(2,3;\,12)$. The principal result is as follows.

\begin{thm} \label{main}
Let $B_0=B_0(2,3;\,12)$. Then the structure of $B_0$ is known. In particular, the following facts hold.
\begin{itemize}
\item $|B_0|=161\,372\,117\,156\,811\,157\,536\,768 =2^{66}\cdot3^7 $.
\item $B_0$ is solvable of derived length $4$ and Fitting length $3$; \ $Z(B_0)=1$.
\item The quotients of the derived series for $B_0$ are $\mathbb{Z}_6$, $\mathbb{Z}_{12}^2$, $\mathbb{Z}_{2}^{61}$, $\mathbb{Z}_{3}^{4}$.
\item A Sylow $2$-subgroup of $B_0$ has nilpotency class $5$ and rank $7$.
\item A Sylow $3$-subgroup of $B_0$ has nilpotency class $2$ and rank $4$.
\item $\operatorname{O}_2(B_0)$ has order $2^{62}$ and nilpotency class $2$; \ $|\Phi(B_0)|=2^{54}$.
\item $\operatorname{O}_{2,3}(B_0)/\operatorname{O}_2(B_0)\cong \mathbb{Z}_3^6$.
\item $B_0/\operatorname{O}_{2,3}(B_0)\cong \rm{SL}_2(3)\circ\mathbb{Z}_4$; in particular, the $3$-length of $B_0$ is $2$.
\item $\operatorname{O}_3(B_0)\cong \mathbb{Z}_3^4$.
\item $\operatorname{O}_{3,2}(B_0)/\operatorname{O}_3(B_0)$ has order $2^{65}$ and nilpotency class $4$.
\item $B_0/\operatorname{O}_{3,2}(B_0)\cong 3^{1+2}:2$; in particular, the $2$-length of $B_0$ is $2$.
\end{itemize}
\end{thm}

By knowing the structure of $B_0$ we mean that this group is constructed explicitly in the computer systems {\textsf{GAP}} \cite{gap} and {\textsf{Magma}} \cite{mag}, and that we are able to prove that this group is indeed $B_0(2,3;\,12)$.
We obtain $B_0$ as a homomorphic image of the finitely presented group
\begin{equation}\label{gd}
G=\langle\, a,b\mid 1=a^2=b^3=w_1^{12}=\ldots =w_{21}^{12}\,\rangle,
\end{equation}
where $w_1,\ldots,w_{21}$ are explicitly given words (\ref{rels}).
The open question of whether $G$ is finite is of interest, because a positive answer would readily imply that $B=B_0$.

A code for \textsf{Magma} that confirms various computation steps made in the paper can be downloaded from \cite{code}.

\section{The algorithm}

This idea is to first construct a certain ``large'' finite (2,3;12)-group and then prove its maximality.
The construction part uses the ``Solvable Quotient'' algorithms available in \textsf{GAP} and \textsf{Magma}.

Let $F=\langle x,y \rangle$ be a free $2$-generator group. The elements of $F$ will be viewed as words in $\{x,y\}$.
Let
\begin{equation}\label{bdef}
B=B(2,3;\,12)=\langle a,b\mid 1=a^2=b^3=w(a,b)^{12}, \ \forall w \in F \rangle
\end{equation}
be the free (2,3;\,12)-group. Clearly, every (2,3;\,12)-group is a quotient of $B$. The first observation is that we can simplify slightly the set of relators for $B$ by leaving only the words $w$ that can be expressed as words in $s=ab$ and $t=ab^2$ and discarding the equivalent words, i.\,e., cyclic permutations and inversions, which do not
change the group. The remaining set of words ordered by length begins as follows
\begin{equation}\label{lst}
L=\{s; st; s^2; s^3t, s^2t^2;  s^4t, s^3t^2, (st)^2s;  \ldots\}
\end{equation}
and may be substituted for $F$ in (\ref{bdef}). The advantage of using the set $L$ in place of $F$ is that it is less redundant and straightforward to construct. We only sketch the proof of the above observation.
\begin{proof}
For every $w\in F$, we have
$$
w(a,b)=b^{\varepsilon_0}ab^{\varepsilon_1}a\ldots ab^{\varepsilon_n} \sim ab^{\varepsilon_1}a\ldots b^{\varepsilon_{n-1}}ab^{\varepsilon_n'}=v(s,t)a^\varepsilon
$$
for some $n\geqslant 0$, where $\varepsilon_0,\varepsilon_n=0,1,2$; $\varepsilon_1,\ldots \varepsilon_{n-1}=1,2$; $\varepsilon_n'\equiv \varepsilon_n +  \varepsilon_0 \pmod{3}$; $v$ is a suitable element of $F$; the symbol ``$\sim$'' means a cyclic permutation; and $\varepsilon=0,1$. The trailing letter ``a'', if present, can be eliminated by a (repeated) cyclic permutation that moves first letter in $v(s,t)$ to the end and then an application of one of the identities
$$
sas=t,\quad tat=s, \quad sat=a, \quad tas=a.
$$
This is always possible, unless we are left with one of the words $a$, $sa$, $ta$, which are equivalent to either $a$ or $b$ and hence can be discarded as well. On the remaining words in $\{s,t\}$ of a given length $l$, a dihedral group $D_{2l}$ acts by effecting cyclic permutations and inversion. The orbit representatives of such actions for $l\geqslant 1$ are collected in~(\ref{lst}).
\end{proof}

We now briefly explain the idea behind the algorithm for finding a large finite $(2,3;\, 12)$-group. Starting off
with the group
$$G=\langle a,b\mid 1=a^2=b^3=r^{12}, \ \forall r \in R \rangle,$$
where $R$ is initially an empty set,
and a known finite $(2,3;\, 12)$-quotient $G_0$ of $G$, say $G_0=\mathbb{Z}_6$, we apply the ``Solvable Quotient'' method to find a bigger quotient $G_1$ of order $2^e\cdot |G_0|$ successively for $e=1,2,\ldots$ until we either succeed or $e$ exceeds the maximal exponent \texttt{max\_e}. If the found quotient has period $12$ then we replace $G_0$ with $G_1$ and start over. Otherwise, we search through the above list $L$ for a word $w$ such that $w^{12}\ne 1$ in $G_1$ and add this word in $R$ to start with a new group $G$. The previously found $G_0$ is still a quotient of $G$, but $G_1$ no longer is, because of the new relator $w^{12}$ for $G$. A more formalized version of this algorithm written in a meta-language is given below.

\medskip

{\small

\begin{itemize}
\item[]\texttt{Input:}
\item[]\quad\verb|l := 10;|$\,//$ Maximal length in the alphabet $\{s,t\}$ of elements in $L$
\item[]\quad\verb|L := [s, s*t, s^2*t, ... ];|$\,//$ Candidates for relators of length at most $l$
\item[]\quad\verb|d0 := 6;|$\,//$ Known order of a (2,3;12)-quotient $G_0$ of $G$
\item[]\quad\verb|R:=[];|$\,//$ Found relators
\item[]\texttt{Multiplier:}
\item[]\quad\verb|p:=2 or 3;  e:=1;|$\,//$ Searching for a new quotient of $G_0$ of order $d_0*p^e$
\item[]\quad\verb|max_e:=10; max_t:=100 h; max_m:=16 G;|$\,//$ Maximal exponent $e$, time, memory
\item[]\texttt{Start:}
\item[]\quad\verb"G := < a,b | 1 = a^2 = b^3 = w^12, w in R >;"
\item[]\texttt{Quotient:}
\item[]\quad\verb|d1 := d0 * p^e;|$\,//$ New order of a quotient to search for
\item[]\quad\verb|G1 := SolvableQuotient(G, d1);|$\,//$ Invoking ``Solvable Quotient'' routine
\item[]\quad\verb|If time > max_t Or memory > max_m  Then -> Output;|
\item[]\quad\verb|If G1 = fail Then { e := e+1; If  e > max_e Then -> Output;|
\item[]\quad\verb|                              Else -> Quotient };|
\item[]\quad\verb|If period of G1 is 12  Then { G0 := G1; d0 := d1; e := 1;|
\item[]\quad\verb|                                             -> Quotient }|
\item[]\quad\verb|Else {|
\item[]\texttt{Search:}
\item[]\quad\verb|find w in L : w(a,b)^12 <> 1 in G1;|
\item[]\quad\verb|If found Then { add w to R; e := 1; -> Start }|
\item[]\quad\verb|Else { l:=l+1;  add words of length l to L; -> Search } };|
\item[]\texttt{Output:}
\item[]\quad\verb|return G0;|
\end{itemize}

}

\medskip

When the search stops (which happens either if the maximal exponent \texttt{max\_e} is exceeded or due to memory/time reasons), we may continue, if necessary, from section \texttt{Multiplier} switching the value of $p$ to $3$, or back to $2$.

This algorithm, despite its limitations, allowed us to construct a large $(2,3;\, 12)$-group. The calculations in \textsf{GAP} yielded a group $G_0$ of order $2^{24}\cdot 3^7$ but exceeded maximal time \texttt{max\_t} when trying to find a quotient of order $2\cdot|G_0|$. The calculations in \textsf{Magma} yielded a group $G_0$ of order $2^{66}\cdot 3^7$
and found no larger quotient of order up to  $2^{10}\cdot|G_0|$, but exceeded maximal memory \texttt{max\_m}
searching for a quotient of order $3\cdot|G_0|$. We assume henceforth that $G_0$ is the latter group of order $2^{66}\cdot 3^7$.

Observe that the set $R$ of found relators consists of the following $21$ words. 
\begin{equation}\label{rels}
\begin{array}{r}
\{\ s,\,st,\,s^2t^2,\,s^4t,\,s^3t,\,s^3t^2,\,s^4t^2,\,  (st)^2s,\,s^5t^2,\,s^3(st)^2,\,s^4(st)^2t,\,s^3(st)^3,\,s^5t^2st,\\[5pt]
s(st)^2t,\,s^4t^4,\, s^2(st)^2t,\,s^2(st^2)^2,\,s^2(st)^2t^2st,\,s^3t^2st,\,s^2(st)^2t^2,\,  s^2(st)^3t \ \},
\end{array}
\end{equation}
where $s=ab$ and $t=ab^2$.

\section{Maximality}

We now address the problem of proving the maximality of the above $(2,3;\, 12)$-group $G_0$ of order $2^{66}\cdot 3^7$. Suppose that there is a larger finite $(2,3;\, 12)$-group $E$. Due to the uniqueness mentioned before Theorem \ref{main}, we may assume that $E$ is a $2,3$-{\it extension}
\begin{equation}\label{ext}
1\to V \to E \stackrel{\pi}{\to} G_0 \to  1,
\end{equation}
which means that the $2,3$-generating pair of $E$ is mapped by $\pi$ to the chosen $2,3$-generating pair of $G_0$.
Moreover, we may assume that $V$ is an elementary abelian $p$-group (with $p=2,3$) and is irreducible as an $\mathbb{F}_pG_0$-module in a natural way. We want to reduce the possibilities for $V$.

Let $p$ a prime. A group $X$ has $p$-{\em period} $p^l$ if every $p$-element of $X$ has order dividing $p^l$. An {\em invariant section} of $X$ is a section of a normal series for $X$. The order of $x\in X$ is denoted by $|x|$.

\begin{lem} \label{redv} Let $H$ be a periodic group of $p$-period $p^l$ and let $V$ be a $p$-elementary abelian invariant section of $H$ viewed as an $\mathbb{F}_pH$-module. Then, for every $h\in H$ of order $p^l$, the element
$$
h_0=1+h+h^2+\ldots+h^{|h|-1}
$$
effects a zero linear transformation of $V$.
\end{lem}
\begin{proof} Let $V=K/N$ for suitable $K,N\trianglelefteqslant H$. Assume to the contrary that $\overline{v}h_0\ne 0$ on for some $\overline{v}\in V$. Observe that $\langle h \rangle\cap K=1$. Indeed, otherwise, the element $h^{p^{l-1}}\in K$ induces the identity transformation of $V$ and $h_0=p\cdot h_1=0$, where $h_1=1+h+\ldots+h^{p^{l-1}-1}$,
contrary to the assumption. Now, let $v\in K$ be a preimage of $\overline{v}$. Since $\langle h \rangle\cap K=1$, we see that $|hv|=|h|\cdot|v_0|$, where
$$
v_0=(hv)^{|h|}=v^{|h|-1}\cdot v^{|h|-2}\cdot \ldots \cdot v.
$$
Clearly, the image of $v_0$ in $V$ is $\overline{v}h_0\ne 0$, which implies that $|v_0|$ is a multiple of $p$. Therefore, $|hv|$ is a multiple of $p^{l+1}$, a contradiction.
\end{proof}

Getting back to the extension (\ref{ext}) of $G_0$, we have the following restriction on the $\mathbb{F}_pG_0$-module $V$.

\begin{cor} $(i)$ In case $p=3$, the action of all elements of $G_0$ of order $3$ on $V$ is {\em quadratic}, i.e. annihilated by the polynomial $x^2+x+1$.

$(ii)$ In case $p=2$, the action of all elements of $G_0$ of order $4$ on $V$ is {\em cubic}, i.e. annihilated by the polynomial $x^3+x^2+x+1$.
\end{cor}

Let $p=2$ and let $\mathcal{X}$ be a representation of $G_0$ corresponding to the module $V$. Then $\mathrm{O}_2(G_0)\leqslant \mathrm{Ker}\mathcal{X}$ and so $V$ is naturally a $G/\mathrm{O}_2(G_0)$-module.
The the quotient $G/\mathrm{O}_2(G_0)$ has small enough order, $2^4\cdot 3^7$, to make it possible to find explicitly all irreducible modules over $\mathbb{F}_2$ and check which of them have cubic action of element of order $4$.
There are five such modules $V_i^{(2)}$, $i=1,\ldots,5$ listed in the first part of Table~\ref{modg0}.

\begin{table} [htb]
\caption{ $\mathbb{F}_pG_0$-modules with quadratic and cubic action \label{modg0} }
\begin{center}
\begin{tabular}{c|ccccc|ccc}
\hline
$p$ &\multicolumn{5}{c|}{$2$} & \multicolumn{3}{c}{$3$} \\[2pt]
\hline
$V$ & $V_1^{(2)}$ & $V_2^{(2)^{\vphantom{A^A}}}$ & $V_3^{(2)}$ & $V_4^{(2)}$ & $V_5^{(2)}$ & $V_1^{(3)}$ & $V_2^{(3)}$ & $V_3^{(3)}$ \\
$\mathrm{dim}\,V$ & $1$ & $2$ & $2$ & $4$ & $6$ & $1$ & $1$ & $4$ \\
absol. irred. & $+$ & $-$ & $+$ & $-$ & $+$ & princ. & $+$ & $-$ \\
$\mathrm{dim}\,H^2(G_0,V)$ & $14$ & $24$ & $12$ & $22$ & $34$ & $3$ & $4$ & $6$ \\
\hline
\end{tabular}
\end{center}
\end{table}

In the case $p=3$, we cannot hope to find all irreducibles for  $G/\mathrm{O}_3(G_0)$ because of its sheer order $2^{66}\cdot 3^3$ and so need another strategy. For $K=\mathbb{F}_3$ and $B=B(2,3;\,12)$, define
\begin{equation}\label{wdef}
W=KB/(x^8+x^4+1\mid x\in B).
\end{equation}
The polynomial $x^8+x^4+1$ here ``encodes'' the quadratic action. Namely, if $|x|$ divides $4$ then $x^8+x^4+1=0$, otherwise, $y=x^4$ has order $3$ and $x^8+x^4+1=y^2+y+1$. Thus,
we may view $W$ as the free cyclic $KB$-module with quadratic action of the elements of $B$ of order $3$. An analog of the following lemma was proved by A.\,S. Mamontov without computer help. 

\begin{lem} \label{lm} The $KB$-module $W$ is finite-dimensional. Namely, $\mathrm{dim}\,W=16$.
\end{lem}
\begin{proof} Define the $2$-generated associative $K$-algebra with $1$
$$
A=\langle\ m,n\ \mid\ 0=m^2-1=n^2+n+1=(mn)^8+(mn)^4+1\ \rangle.
$$
Clearly, as a $K$-algebra, $W$ is a homomorphic image of $A$ under the homomorphism that extends the map $\varphi: (m,n)\mapsto(\overline{a},\overline{b})$, where $(\overline{a},\overline{b})$ is the image in $W$
of the generating 2,3-pair $(a,b)$ of $B$. We prove that $\varphi$ is in fact an isomorphism. Using the \textsf{Magma} method ``FPAlgebra'' for constructing finitely presented algebras it is readily verified that $\mathrm{dim}\,A=16$. Now, the elements $m,n\in A$ are invertible and the group $A_0=\langle m,n\rangle$ can be checked to have order $432$, exponent $12$, and quadratic action on $A$ of the elements of order $3$. This implies that $A_0$ is a homomorphic image of $B$ and $A$ is a (cyclic) $KB$-module with quadratic action. Since $W$ is a free $KB$-module with these properties, there exists a homomorphism $W\to A$ which is clearly an inverse of~$\varphi$.
\end{proof}

As every irreducible module is cyclic, we conclude that $V$ is an irreducible homomorphic image of the free module $W$. It can be checked that up to isomorphism $W$ has only three irreducible factors $V_i^{(3)}$, $i=1,2,3$, which are listed in the second part of Table \ref{modg0}.

We remark that there is no known analog of Lemma \ref{lm} that would help to find the modules with cubic action in characteristic 2.

\section{2,3-extensions}

Now that we have restricted the possibilities for $V$ in (\ref{ext}), we need to check, for every possible extension $E$ with a given $V$, if $E$ is $2,3$-generated of period $12$. Since $E$ can be constructed as a polycyclic group and {\textsf{Magma}} is efficient in calculating the exponent of a polycyclic group, the verification of whether $E$ has period $12$ presents no problem (in fact, if $E$ is a {\em split} extension, it will automatically have period $12$ due to the restrictions on $V$). Therefore, we will concentrate on finding sufficient conditions for $E$ to be a $2,3$-extension of $G_0$.

\begin{lem} \label{cri} Let the extension $(\ref{ext})$ be nonsplit, where $G_0=\langle a, b \rangle$ and $V$ is irreducible. Then $E$ is a $|a|,|b|$-extension if and only if there are preimages under $\pi$ of $a$ and $b$  of orders $|a|$ and $|b|$, respectively.
\end{lem}
\begin{proof} Let $\hat{a}$, $\hat{b}$ be such preimages and let $\hat{E}=\langle \hat{a},\hat{b}\rangle$. We have $E=\hat{E}V$ and $\hat{E}\cap V\ne 0$, since $E$ is nonsplit. But $V$ is irreducible, which yields $\hat{E}\cap V=V$
and so $\hat{E}=E$. The converse holds by definition.
\end{proof}

There are different ways to check whether $\pi^{-1}(a)$ contains an element of order $|a|$. Since $|\pi^{-1}(a)|=|V|$, for small modules $V$, one could use exhaustive search through all preimages. However, the following is a more conceptual approach which works for larger modules, too.

Let $\tau :G_0\to E$ be a transversal in (\ref{ext}), i.\,e., $\pi\circ \tau=\mathrm{id}_{G_0}$. If we choose $\tau$ such that $\tau(1)=1$ then the corresponding $2$-cocycle $\gamma: G_0\times G_0\to V$, which is defined by
\begin{equation}\label{coc}
\gamma(g_1,g_2)=\tau(g_1g_2)^{-1}\tau(g_1)\tau(g_2),\quad g_1,g_2\in G_0,
\end{equation}
will be {\em normalized}, i.\,e. $\gamma(1,1)=0$. (We use additive notation in $V$.) The set $Z_N^2(G_0,V)$ of all normalized $2$-cocycles is a subspace in $Z^2(G_0,V)$. We also define $B_N^2(G_0,V)=Z_N^2(G_0,V)\cap B^2(G_0,V)$ to be
the set of normalized $2$-coboundaries.

\begin{lem} \label{prei} Let $G_0$ be a finite group, let $V$ be a $KG_0$-module over a field $K$, and let an extension $(\ref{ext})$ be defined by a $2$-cocycle $\gamma\in Z_N^2(G_0,V)$. For $g\in G_0$, define
\begin{align*}
\psi_g(\gamma)&=\gamma(g,g)+\gamma(g,g^2)+\ldots+\gamma(g,g^{|g|-1}),\\
g_0&=1+g+\ldots+g^{|g|-1}.
\end{align*}
Then
\begin{enumerate}
\item[$(i)$] $\exists\  \hat{g}\in\pi^{-1}(g)\, :\  |\hat{g}|=|g|\quad \Leftrightarrow \quad \psi_g(\gamma)\in\mathrm{Im}(g_0)$.
\item[$(ii)$] $\forall\  \hat{g}\in\pi^{-1}(g)\quad |\hat{g}|=|g|\quad \Leftrightarrow \quad \psi_g(\gamma)=0\ $ and $\ \mathrm{Im}(g_0)=0$.

\item[$(iii)$] If $\ \mathrm{Im}(g_0)=0$ then $B_N^2(G_0,V)\subseteq \mathrm{Ker}\, \psi_g$. In particular, $\psi_g$ induces a $K$-linear map  $\overline{\psi}_g:H^2(G_0,V)\to V$.

\end{enumerate}
\end{lem}
\begin{proof} Let $\tau :G_0\to E$ be a transversal in (\ref{ext}) such that $\tau(1)=1$. Then $\pi^{-1}(g)=\{\tau(g)v\mid v\in V\}$. We see that $(\tau(g)v)^{|g|}=\tau(g)^{|g|}+vg_0$ is the zero of $V$ if and only if $\tau(g)^{|g|}=-vg_0$. An inductive application of (\ref{coc}) yields $\tau(g)^{|g|}=\tau(1)\psi_g(\gamma)=\psi_g(\gamma)$. These remarks imply $(i)$ and $(ii)$.

We now prove $(iii)$. Let $\gamma_f\in B_N^2(G_0,V)$, where $f:G\to V$ satisfies $f(1)=0$. This means that $\gamma_f(g_1,g_2)=f(g_1g_2)-f(g_1)\cdot g_2-f(g_2)$ for all $g_1,g_2\in G$. We have
$$
\begin{array}{rl}
\gamma_f(g,g)&=f(g^2)-f(g)\cdot g-f(g),\\[3pt]
\gamma_f(g,g^2)&=f(g^3)-f(g)\cdot g^2-f(g^2),\\
&\cdots\\
\gamma_f(g,g^{|g|-1})&=f(1)-f(g)\cdot g^{|g|-1}-f(g^{|g|-1}).\\[3pt]
\end{array}
$$
Summing up gives $\psi_g(\gamma_f)=-f(g)\cdot g_0 = 0$, because $\mathrm{Im}(g_0)=0$. The claim follows.
\end{proof}

The meaning of Lemmas \ref{cri} and \ref{prei}$(i)$   is that they reduce checking the $2,3$-generation of nonsplit extensions to a linear calculation in the module $V$. Here is an analogous result for split extensions.

\begin{lem} \label{spl} Let $G_0=\langle a,b\rangle$ be a finite group and let $V$ be an irreducible finite-dimensional $KG_0$-module over a field $K$. Denote by $E$ the natural semidirect product of $G_0$ and $V$, and set $a_0=1+a+\ldots a^{|a|-1}$, $b_0=1+b+\ldots b^{|b|-1}$.  Then $E$ is an $|a|,|b|$-extension of $G_0$ if and only if
$$
\mathrm{dim}\,\mathrm{Ker}\,a_0 + \mathrm{dim}\,\mathrm{Ker}\, b_0 -\mathrm{dim}\,V > \mathrm{dim}\,H^1(G_0,V) - \mathrm{dim}\,H^0(G_0,V).
$$
\end{lem}
\begin{proof} For $v\in V$, we have $(av)^{|a|}=va_0$, which is the zero of $V$ if and only if $v\in \mathrm{Ker}\,a_0$, and similarly for $b$. Hence, the number of $|a|,|b|$-pairs of elements of $E$ that cover $(a,b)$  is $|\mathrm{Ker}\,a_0|\cdot|\mathrm{Ker}\,b_0|$. If $(a_1,b_1)$ is such a pair then $\langle a_1,b_1 \rangle V=E$ and, due to the irreducibility of $E$, we see that $\langle a_1,b_1 \rangle$ either equals $E$ or is a complement to $V$. Conversely, every complement $G_1$ uniquely determines an $|a|,|b|$-pairs above $(a,b)$. Hence, the number of generating $|a|,|b|$-pairs is the difference between the number of all pairs and the number of complements. Observe that the number of complements that are conjugate to a fixed one, $G_1$, equals
$$
|E:N_E(G_1)|=|V:C_V(G_1)|=|V:C_V(G_0)|=|V|/|H^0(G_0,V)|
$$
and does not depend on the complement. Also, it is well known that there are $|H^1(G_0,V)|$ conjugacy classes of complements to $V$ in $E$. Hence, there are a total of
$$|V|\cdot|H^1(G_0,V)|/|H^0(G_0,V)|$$
complements. Passing to dimensions, we obtain the required inequality.
\end{proof}

\section{The search}
The results in the previous section can be used to check the $2,3$-generation of
all extensions (\ref{ext}) with a given irreducible module $V$ by running through the elements of $H^2(G_0,V)$, unless the order $|H^2(G_0,V)|$ is too big. This was done for all modules $V$ from Table
\ref{modg0} except for $V_2^{(2)}$ and $V_5^{(2)}$, and no $2,3$-extensions of $G_0$ of period $12$ were found.
For the excluded two modules, this method would require searching through as many as $2^{24}$ and $2^{36}$ extensions, respectively. In these cases, we do the elimination in a different way, which we briefly explain.

Let $\mathcal{X}$ be the representation corresponding to one of the excluded modules $V$. Observe that the element $g_1 = ab\in G_0$ has order $12$. Hence, a necessary condition for $E$ to be of period $12$ is that all elements in $\pi^{-1}(g_1)$ have order $12$. Since $\mathcal{X}(g_0)=0$, where $g_0=1+g_1+\ldots+g_1^{|g_1|-1}$, Lemma \ref{prei}$(ii),(iii)$ implies that the extensions $E$ satisfying this condition are defined by the elements of $\mathrm{Ker}\,\overline{\psi}_{g_1}$. This kernel may turn out to be a proper subspace of $H^2(G_0,V)$ thus reducing the dimension of the space to search in. We may repeat this procedure by taking a new element $g_2\in G_0$ of order $12$ or $4$ which should be, in a sense, ``independent'' of $g_1$ and attempt
to reduce the dimension further. It turns out that the elements $g_i$ of $L$ in (\ref{lst}), which we used as candidates for relators, are also good candidates for such independent elements of $G_0$. As a result, we found that,
as $g_i$ runs through the first few ($\leqslant 20$) such elements of $L$,
the dimension of $\cap_{i}\mathrm{Ker}\,\overline{\psi}_{g_i}$ is $2$ for $V=V_2^{(2)}$ and $0$ for $V=V_5^{(2)}$.
This left us to consider only the split extensions for both modules and three inequivalent (but isomorphic) nonsplit extensions for $V=V_2^{(2)}$. All these, despite having period $12$, were found not to be $2,3$-generated. This final elimination proves that $G_0$ is in fact $B_0$ as claimed in Theorem \ref{main}.

\medskip

{\em Acknowledgement.} The author is thankful to Prof. V.\,D. Mazurov and Dr. A.\,S. Mamontov for a useful discussion of this paper, and to anonymous referees whose comments resulted in an improvement to the original version.

\end{document}